\documentclass{article}
\usepackage{amsmath}
\usepackage{amssymb}
\usepackage{amsthm}
\usepackage{fullpage}
\usepackage{cleveref}
\usepackage[utf8]{inputenc}
\usepackage{framed}
\usepackage{enumerate}
\usepackage{xcolor}

\title{From DNF compression to sunflower theorems via regularity}
\author{
Shachar Lovett\thanks{Supported by NSF grant CCF-1614023.}\\
University of California, San Diego\\
\texttt{slovett@ucsd.edu}
\and
Noam Solomon\\
MIT\\
\texttt{noam.solom@gmail.com}
\and
Jiapeng Zhang\thanks{Supported by NSF grant CCF-1614023.}\\
University of California, San Diego\\
\texttt{jpeng.zhang@gmail.com}
}

\newtheorem{theorem}{Theorem}[section]
\newtheorem{claim}[theorem]{Claim}
\newtheorem{corollary}[theorem]{Corollary}
\newtheorem{lemma}[theorem]{Lemma}
\newtheorem{definition}[theorem]{Definition}
\newtheorem{conjecture}[theorem]{Conjecture}
\newtheorem{problem}[theorem]{Problem}
\newtheorem{example}[theorem]{Example}
\newtheorem{remark}[theorem]{Remark}

\newcommand{\F}{\mathbb{F}}
\newcommand{\FF}{\mathcal{F}}

\newcommand{\PP}{\mathcal{P}}
\newcommand{\eps}{\varepsilon}

\newcommand{\D}{\mathcal{D}}

\begin{document}

\maketitle

\begin{abstract}
The sunflower conjecture is one of the most well-known open problems in combinatorics. It has several applications
in theoretical computer science, one of which is DNF compression, due to Gopalan, Meka and Reingold (Computational Complexity, 2013). In this paper, we show that improved bounds for DNF compression
imply improved bounds for the sunflower conjecture, which is the reverse direction of the DNF compression result. The main approach is based on regularity of set systems
and a structure-vs-pseudorandomness approach to the sunflower conjecture.
\end{abstract}

\section{Introduction}
The sunflower conjecture is one of the most well-known open problems in combinatorics. An $r$-sunflower is a family of $r$ sets $S_1,\ldots,S_r$
where all pairwise intersections are the same. A $w$-set system is a collection of sets where each set has size at most $w$.
Erd\H{o}s and Rado \cite{ErdosR1960} asked how large can a $w$-set system be, without containing an $r$-sunflower.
They proved an upper bound of $w! (r-1)^w$, and conjectured that the bound can be improved.
\begin{conjecture}[Sunflower conjecture, \cite{ErdosR1960}]
\label{conj:sunflower}
Let $r \ge 3$. There is a constant $c_{r}$ such that any $w$-set system $\FF$ of size $|\FF|\geq c_{r}^{w}$ contains an $r$-sunflower.
\end{conjecture}
60 years later, only lower order improvements have been achieved, and the best bounds are still of the order of magnitude of about  $w^w$ for any fixed $r$, same as in the original theorem of Erd\H{o}s and Rado. A good survey on the current bounds is \cite{kostochka2000extremal}.

Sunflowers have been useful in various areas in theoretical computer science.
Some examples include monotone circuit lower bounds  \cite{razborov1985lower,rossman2010monotone}, barriers for improved algorithms for matrix multiplication \cite{alon2013sunflowers}
and faster deterministic counting algorithms via DNF compression \cite{gopalan2013dnf}. The focus on this paper is on this latter application, in particular DNF compression.

A DNF (Disjunctive Normal Form) is disjunction of conjunctive terms. The \emph{size} of a DNF is the number of terms, and the \emph{width} of a DNF is the maximal number of literals in a term.
It is a folklore result that any DNF of size $s$ can be approximated by another DNF of width $O(\log s)$, by removing all terms of larger width. The more interesting direction
is whether DNFs of small width can be approximated by DNFs of small size. Namely - can DNFs of small width be ``compressed" while approximately preserving their computational structure?

A beautiful result of Gopalan, Meka and Reingold \cite{gopalan2013dnf} shows that DNFs of small width can be approximated by small size DNFs. Their proof relies on the sunflower theorem (more precisely,
a variant thereof due to Rossman \cite{rossman2010monotone} that we will discuss shortly). Before stating their result, we introduce some necessary terminology. We say that two functions $f,g:\{0,1\}^n \to \{0,1\}$ are $\eps$-close if $\Pr[f(x) \ne g(x)] \le \eps$ over a uniformly chosen input. We say that $f$ is a lower bound of $g$, or that $g$ is an upper bound of $f$, if $f(x) \le g(x)$ for all $x$.

\begin{theorem}[DNF compression using sunflowers, sandwiching bounds \cite{gopalan2013dnf}]
\label{thm:GMR}
Let $f$ be a width-$w$ DNF. Then for every $\eps>0$ there exist two width-$w$ DNFs, $f_{lower}$ and $f_{upper}$ such that
\begin{enumerate}[(i)]
\item $f_{lower}(x) \le f(x) \le f_{upper}(x)$ for all $x$.
\item $f_{lower}$ and $f_{upper}$ are $\eps$-close.
\item $f_{lower}$ and $f_{upper}$ have size $(w \log(1/\eps))^{O(w)}$.
\end{enumerate}
\end{theorem}

Recently, Lovett and Zhang \cite{Lovett2018DNFSB} improved the dependence of the size of the lower bound DNF on $w$
(but with a worse dependence on $\eps$). In particular, the proof avoids the use of the sunflower theorem.

\begin{theorem}[DNF compression without sunflowers, lower bound \cite{Lovett2018DNFSB}]
\label{thm:lower_bound_DNF}
Let $f$ be a width-$w$ DNF. Then for every $\eps>0$ there exists a width-$w$ DNFs $f_{lower}$ such that
\begin{enumerate}[(i)]
\item $f_{lower}(x) \le f(x)$ for all $x$.
\item $f_{lower}$ and $f$ are $\eps$-close.
\item $f_{lower}$ has size $(1/\eps)^{O(w)}$.
\end{enumerate}
\end{theorem}

It is natural to speculate that a similar bound holds for upper bound DNFs.

\begin{conjecture}[Improved upper bound DNF compression]
\label{conj:upper_bound_DNF_nonmon}
Let $f$ be a width-$w$ DNF. Then for every $\eps>0$ there exists a width-$w$ DNF $f_{upper}$ such that
\begin{enumerate}[(i)]
\item $f(x) \le f_{upper}(x)$ for all $x$.
\item $f_{upper}$ and $f$ are $\eps$-close.
\item $f_{upper}$ has size $(1/\eps)^{O(w)}$.
\end{enumerate}
\end{conjecture}

To study the connection between DNF compression and sunflowers, we would need an analog of \Cref{conj:upper_bound_DNF_nonmon} for monotone DNFs.

\begin{conjecture}[Improved upper bound monotone DNF compression]
\label{conj:upper_bound_DNF}
In \Cref{conj:upper_bound_DNF_nonmon}, if $f$ is a monotone DNF, then $f_{upper}$ can also be taken to be a monotone DNF.
\end{conjecture}

The main result of this paper is that \Cref{conj:upper_bound_DNF} implies an improved bound for the sunflower conjecture, with a bound of $(\log w)^{O(w)}$ instead of the current bound of $O(w)^{w}$. Thus, the connection between sunflower theorems and DNF compression goes both ways. We note that the proof of \cite{gopalan2013dnf} is also true for monotone DNF compression.

To simplify the presentation, we assume from now on that $w \ge 2$. This will allow us to assume that $\log w >0$. In any case, for $w=1$ the sunflower conjecture is trivial, as any $1$-set system of size $r$ is an $r$-sunflower.

\begin{theorem}[Main theorem]
\label{thm:main}
Assume that \Cref{conj:upper_bound_DNF} holds. Then for any $r \ge 3$ there exists a constant $c_r$ such that the following holds. Any $w$-set system $\FF$ of size $|\FF| \ge (\log w)^{c_r w}$ contains an $r$-sunflower.
\end{theorem}

In fact, \Cref{thm:main} holds even with a slightly weaker conjecture instead of \Cref{conj:upper_bound_DNF}, where
the size bound can be assumed to be $((\log w) / \eps)^{O(w)}$ instead of $(1/\eps)^{O(w)}$.

\subsection{Proof overview}

The proof of Erd\H{o}s and Rado \cite{ErdosR1960} is by a simple case analysis which we now recall. Let $\FF$ be a $w$-set system. Then either $\FF$ contains $r$ pairwise disjoint sets, which are in particular an $r$-sunflower; or at most $r-1$ sets whose union intersects all other sets. In the latter case, there is an element that belongs to a $\frac{1}{(r-1)w}$ fraction of the sets in $\FF$. If we restrict to these sets, and remove the common element, then we reduced the problem to a $(w-1)$-set system of size $\frac{|\FF|}{(r-1)w}$. The proof concludes by induction.

Our approach is to refine this via a structure-vs-pseudorandomness approach. Either there is a set $T$ of elements that belong to many sets in $\FF$ (concretely, at least $|\FF|/\kappa^{|T|}$, for an appropriately chosen $\kappa$), or otherwise the set system $\FF$ is pseudo-random, in the sense that no set $T$ is contained in too many sets in $\FF$. The main challenge is showing that by choosing $\kappa$ large enough, this notion of pseudo-randomness is useful. This will involve introducing several new concepts and tying them to the sunflower problem.

The following proof overview follows the same structure as the sections in the paper, to ease readability.

\paragraph{\Cref{sec:dnfs_set_systems}: DNFs and set systems.}
First, we note that set systems are one-to-one correspondence to monotone DNFs.
Formally, we identify a set system $\FF=\{S_1,\ldots,S_m\}$ with the monotone DNF
$f_{\FF}(x) = \bigvee_{S \in \FF} \bigwedge_{i \in S} x_i$. This equivalence will be useful in the proof, as at different stages one of these viewpoints is more convenient.

The notions of ``lower bound DNF" $f_{lower}$ and ``upper bound DNF" $f_{upper}$ used in
\Cref{thm:GMR}, \Cref{thm:lower_bound_DNF} and \Cref{conj:upper_bound_DNF} have analogs for
set systems, which we refer to as \emph{proper} lower bound and upper bound DNFs (or set systems). For the purpose of this high level overview, we ignore this distinction here.

\paragraph{\Cref{sec:approx_sunflowers}: Approximate sunflowers.}
The notion of approximate sunflowers was introduced by Rossman \cite{rossman2010monotone}. It relies on the notion of \emph{satisfying set systems}.

Let $\FF$ be a set system on a universe $X$. We say that $\FF$ is \emph{$(p,\eps)$-satisfying} if $\Pr_{x \sim X_p}[f_{\FF}(x)=1] > 1-\eps$,
where $f_{\FF}$ is the corresponding monotone DNF for $\FF$, and $X_p$ is the $p$-biased distribution on $X$. The importance of satisfying set systems in our context is that a $(1/r,1/r)$-satisfying set system contains $r$ pairwise disjoint sets (\Cref{claim:satisfying_contains_disjoint}).

Let $K = \cap_{S \in \FF} S$ be the intersection of all sets in $\FF$. We say that $\FF$ is a \emph{$(p,\eps)$-approximate sunflower} if the set system $\{S \setminus K: S \in \FF\}$ is $(p,\eps)$-satisfying. An interesting connection between approximate sunflowers and sunflowers is that a $(1/r,1/r)$-approximate sunflower contains an  $r$-sunflower (\Cref{cor:approx_contains_sunflower}).

\paragraph{\Cref{sec:regular}: Regular set systems.}
Let $\D$ be a distribution over subsets of $X$. We say that $\D$ is \emph{regular} if when sampling $S \sim \D$, the probability that $S$ contains any given set $T$ is exponentially small in the size of $T$. Formally, $\D$ is $\kappa$-regular if for any set $T \subseteq X$ it holds that $\Pr_{S \sim \D}[T \subseteq S]  \le \kappa^{-|T|}$.

A set system $\FF$ is $\kappa$-regular if there exists a $\kappa$-regular distribution supported on sets in $\FF$. We show that if $\FF$ is $\kappa$-regular, then the same holds for any upper bound set system (\Cref{claim:regular_upper}) and any ``large enough" lower bound set system (\Cref{claim:regular_lower}). These facts will turn out to be useful later.

\paragraph{\Cref{sec:regular_satisfying}: Regular set systems are $(1/2,1/2)$-satisfying.}
In this section, we focus on regular set systems $\FF$, or equivalently regular DNFs $f=f_{\FF}$.
We show that, assuming \Cref{conj:upper_bound_DNF} (or the slightly weaker \Cref{conj:upper_bound_weak}), any $\kappa$-regular DNF of width $w$,
where $\kappa = (\log w)^{O(1)}$, is $(1/2,1/2)$-satisfying. Namely, $\Pr[f(x)=1] \ge 1/2$,
where $x$ is uniformly chosen. In particular, this implies that $\FF$ contains two disjoint sets. However, our goal is to prove that $\FF$ contains an $r$-sunflower for $r \ge 3$, so we are not done yet.

\paragraph{\Cref{sec:intersecting}: Intersecting regular set systems.}
Let $\alpha(w,r)$ denote the maximal $\kappa$ such that there exists a $\kappa$-regular $w$-set system without $r$ pairwise disjoint sets. It is easy to prove that the sunflower theorem holds for any set system of size $|\FF| > \alpha(w,r)^w$ (\Cref{claim:alpha_w_r_sunflower}). However, our discussion so far only allows us to bound $\beta(w)=\alpha(w,2)$; concretely, assuming \Cref{conj:upper_bound_DNF} we have $\beta(w) \le (\log w)^{O(1)}$.

We show (\Cref{lemma:alpha_beta}) that nontrivial upper bounds on $\beta(w)$ imply related upper bounds on $\alpha(w,r)$ for every $r$. Concretely, if $\beta(w) \le (\log w)^{O(1)}$ then
$\alpha(w,r) \le (\log w)^{c_r}$ where $c_r>0$ are constants. This concludes the proof, as we get that any $w$-set system of size $|\FF| \ge (\log w)^{c_r w}$ must contain an $r$-sunflower.

\paragraph{Acknowledgements.} We thank Ray Li for spotting a subtle mistake in the previous version (its solution necessitated restricting some claims to non-trivial or non-redundant set systems). We also thank CCC reviewers for pointing out some early mistakes.

\section{DNFs and set systems}
\label{sec:dnfs_set_systems}

A DNF is \emph{monotone} if it contains no negated variables. Monotone DNFs are in one-to-one correspondence with set systems. Formally, if $\FF$ is a set system then the corresponding monotone DNF is
$$
f_{\FF}(x) = \bigvee_{S \in \FF} \bigwedge_{i \in S} x_i.
$$
In the other direction, if $f=\bigvee_{j \in [m]} \bigwedge_{i \in S_j} x_i$ is a monotone DNF then its corresponding set system is
$$
\FF_{f} = \{S_1,\ldots,S_m\}.
$$
Observe that a $w$-set system corresponds to a width-$w$ monotone DNF, and vice versa. If $X$ is the set of elements over which
$\FF$ is defined then we write $\FF \subseteq \PP(X)$.

A DNF is \emph{non-redundant} if no term implies another term. A DNF is \emph{non-trivial} if it is not a constant function. This motivates the following definitions for the corresponding set systems.

\begin{definition}[Non-redundant set systems]
A set system $\FF$ is \emph{non-redundant} if it is an anti-chain. Namely, it does not contain two distinct sets $S_1,S_2$ with $S_1 \subset S_2$.
\end{definition}

\begin{definition}[Non-trivial set systems]
A set system $\FF$ is \emph{non-trivial} if it is not empty, and doesn't contain the empty set.
\end{definition}

To recall, we consider both lower bound and upper bound DNFs. As our main motivation is to better understand sunflowers, we
restrict attention to monotone DNFs from now on; however, all the definitions can be easily adapted for general DNFs.

We next define \emph{proper} upper and lower bound DNFs. Proper lower bound DNFs are obtained by removing terms from the DNF, and
proper upper bound DNFs are obtained by removing variables from terms in the DNF.
We describe both in terms of the corresponding set systems.

\begin{definition}[Proper lower bound DNF / set system]
Let $\FF$ be a set system. A \emph{proper lower bound set system} for $\FF$ is simply a sub set system $\FF' \subseteq \FF$.
Observe that indeed
$$
f_{\FF'}(x) \le f_{\FF}(x) \qquad \forall x.
$$
\end{definition}

\begin{definition}[Proper upper bound DNF / set system]
Let $\FF$ be a set system. A \emph{proper upper bound set system} for $\FF$ is a set system $\FF'$ that satisfies the following: for each $S \in \FF$
there exists $S' \in \FF'$ such that $S' \subseteq S$.
Observe that indeed
$$
f_{\FF'}(x) \ge f_{\FF}(x) \qquad \forall x.
$$
\end{definition}

For monotone DNFs, upper bounds and proper upper bounds are the same.

\begin{claim}
Let $\FF,\FF'$ be set systems over the same universe, such that
$$
f_{\FF'}(x) \ge f_{\FF}(x) \qquad \forall x.
$$
Then $\FF'$ is a proper upper bound set for $\FF$.
\end{claim}

\begin{proof}
Assume not.
Then there exists $S \in \FF$ such that there is no $S' \in \FF'$ with $S' \subseteq S$. Let $x=1_S$ be the indicator vector for $S$. Then $f_{\FF}(x)=1$ but $f_{\FF'}(x)=0$, a contradiction.
\end{proof}

\begin{corollary}
In \Cref{conj:upper_bound_DNF}, we may assume that $f_{upper}$ is a proper upper bound DNF for $f$.
\end{corollary}

We note that the lower and upper bound DNFs in \cite{gopalan2013dnf} are in fact proper lower and upper bounds, and the same holds for the lower bound DNF in \cite{Lovett2018DNFSB}.

\section{Approximate sunflowers}
\label{sec:approx_sunflowers}

We introduce the notion of \emph{approximate sunflowers}, first defined by Rossman \cite{rossman2010monotone}.
We first need some notation. Given a finite set $X$ and $0<p<1$,
we denote by $X_p$ the \emph{$p$-biased} distribution over $X$, where $W \sim X_p$ is sampled by including each $x \in X$ in $W$ independently with probability $p$. The definition of approximate sunflowers relies on the notion of a \emph{satisfying set system}.

\begin{definition}[Satisfying set system] Let $\FF \subseteq \PP(X)$ be a set system and let $0<p,\eps<1$.
We say that $\FF$ is $(p,\eps)$-satisfying if
$$
\Pr_{W \sim X_p}\left[\exists S \in \FF: S \subseteq W\right]> 1-\eps.
$$
\end{definition}

Equivalently, if $f_{\FF}:\{0,1\}^X \to \{0,1\}$ is the DNF corresponding to $\FF$, then $\FF$ is $(p,\eps)$-satisfying if
$$
\Pr_{x \sim X_p}[f_{\FF}(x)=1] > 1-\eps.
$$

An approximate sunflower is a set system which is satisfying if we first remove the common intersection of all the sets in the set system.

\begin{definition}[Approximate sunflower]
Let $\FF \subseteq \PP(X)$ be a set system and let $0<p,\eps<1$. Let $K = \cap_{S \in \FF} S$. Then $\FF$ is a  $(p,\eps)$-approximate sunflower
if the set system $\{S \setminus K: S \in \FF\}$ is $(p,\eps)$-satisfying.
\end{definition}

Rossman proved an analog of the sunflower theorem for approximate sunflowers. Li, Lovett and Zhang \cite{li2018sunflowers} reproved this theorem by using a connection to randomness extractors.

\begin{theorem}[Approximate sunflower lemma \cite{rossman2010monotone}]
Let $\FF$ be a $w$-set system and let $\eps>0$. If
$|\FF|\geq w!\cdot(1.71\log(1/\eps)/p)^{w}$ then $\FF$ contains a $(p, \eps)$-approximate sunflower.
\end{theorem}

To conclude this section, we show that satisfying set systems contain many disjoint sets, and hence approximate sunflowers contain sunflowers.

\begin{claim}
\label{claim:satisfying_contains_disjoint}
Let $\FF$ be a non-trivial set system, $r \ge 2$, and assume that $\FF$ is a $(1/r,1/r)$-satisfying. Then $\FF$ contains $r$ pairwise disjoint sets.
\end{claim}

\begin{proof}
Let $\FF \subseteq \PP(X)$. Consider a uniform random coloring of $X$ with $r$ colors.
A coloring induces a partition of $X$ into $X=W_1 \cup \ldots \cup W_r$, where $W_c$ is the set of all elements that attain the color $c$.
Given a color $c \in [r]$, a set $S \in \FF$ is $c$-monochromatic if all its elements attain the color $c$.
Observe that for each color $c$,
$$
\Pr[\exists S \in \FF, \; S \text{ is } c\text{-monochromatic}] = \Pr[\exists S \in \FF, S \subseteq W_c].
$$
The marginal distribution of each $W_c$ is $(1/r)$-biased. By our assumption that $\FF$ is $(1/r,1/r)$-satsifying, the probability that $W_c$ contains some $S \in \FF$
is more than $1-1/r$. So by the union bound,
$$
\Pr[\forall c \in [r] \;\exists S \in \FF, \;S \text{ is } c\text{-monochromatic}] > 0.
$$
In particular, there exists a coloring where this event happens. Let $S_1,\ldots,S_r$ be the sets for which $S_c$ is $c$-monochromatic.
As $\FF$ is non-trivial, $S_1,\ldots,S_r$ are non-empty sets, and hence must be distinct.
Thus $S_1,\ldots,S_r$ must be pairwise disjoint.
\end{proof}

\begin{corollary}
\label{cor:approx_contains_sunflower}
Let $\FF$ be a non-redundant set system with $|\FF| \ge 2$, $r \ge 2$, and assume that $\FF$ is a $(1/r,1/r)$-approximate sunflower. Then $\FF$ contains an $r$-sunflower.
\end{corollary}

\begin{proof}
Let $K=\cap_{S \in \FF} S$ and define $\FF' = \{S \setminus K: S \in \FF\}$ which by assumption is $(1/r,1/r)$-satisfying.
As $\FF$ is non-redundant and has at least two elements, $\FF'$ is non-trivial. By \Cref{claim:satisfying_contains_disjoint}
$\FF'$ contains $r$ pairwise disjoint sets $S_1 \setminus K, \ldots, S_r \setminus K$. This implies that $S_1,\ldots,S_r$ form an $r$-sunflower.
\end{proof}

\section{Regular set systems}
\label{sec:regular}

The notion of regularity of a set system is pivotal in this paper. At a high level, a set system is regular if no element
belongs to too many sets, no pair of elements belongs to too many sets, and so on. It is closely related to the notion
of block min-entropy studied in the context of lifting theorems in communication complexity \cite{goos2016rectangles}.

\begin{definition}[Regular distribution]
Let $X$ be a finite set, and let $\D$ be a distribution on non-empty subsets $S \subseteq X$.
The distribution $\D$ is $\kappa$-regular if for any set $T \subseteq X$ it holds that
$$
\Pr_{S \sim \D}[T \subseteq S] \le \kappa^{-|T|}.
$$
\end{definition}

\begin{remark}
Note that we need to restrict $\D$ to be supported on non-empty sets, as otherwise the regularity can be infinite.
\end{remark}

\begin{definition}[Regular set system]
A non-trivial set system $\FF$ is $\kappa$-regular if there exists a $\kappa$-regular distribution $\D$
supported on the sets in $\FF$.
\end{definition}

The following claims show that if $\FF$ is a $\kappa$-regular set system then any proper upper bound set system for it is also
$\kappa$-regular, and any ``large" proper lower bound set system is approximately $\kappa$-regular.

\begin{claim}
\label{claim:regular_upper}
Let $\FF$ be a non-trivial $\kappa$-regular set system. Let $\FF'$ be a non-trivial proper upper bound set system for $\FF$. Then $\FF'$ is also $\kappa$-regular.
\end{claim}

\begin{proof}
Let $\D$ be a $\kappa$-regular distribution supported on $\FF$. Let $\varphi:\FF \to \FF'$ be a map such that $\varphi(S) \subseteq S$
for all $S \in \FF$. Define a distribution $\D'$ on $\FF'$ as follows:
$$
\D'(S') = \sum_{S \in \varphi^{-1}(S')} \D(S).
$$
Then for any set $T$,
$$
\Pr_{S' \sim \D'}[T \subseteq S'] = \sum_{S' \in \FF': T \subseteq S'} \D'(S')
= \sum_{S \in \FF: T \subseteq \varphi(S)} \D(S)
\le \sum_{S \in \FF: T \subseteq S} \D(S) = \Pr_{S \sim \D}[T \subseteq S] \le \kappa^{-|T|}.
$$
\end{proof}

\begin{claim}
\label{claim:regular_lower}
Let $\FF$ be a non-trivial $\kappa$-regular set system, and $\D$ be a $\kappa$-regular distribution supported on $\FF$.
Let $\FF' \subseteq \FF$ be a non-trivial proper lower bound set system for $\FF$, and let $\alpha=\D(\FF')$.
Then $\FF'$ is $(\kappa \alpha)$-regular.
\end{claim}

\begin{proof}
Define a distribution $\D'$ on $\FF'$ by $\D'(S) = \alpha^{-1} \D(S)$. Then for any non-empty set $T$,
$$
\Pr_{S \sim \D'}[T \subseteq S] \le \alpha^{-1} \Pr_{S \sim \D}[T \subseteq S] \le \alpha^{-1} \kappa^{-|T|}
\le (\kappa \alpha)^{-|T|}.
$$
\end{proof}

\section{Regular set systems are $(1/2,1/2)$-satisfying}
\label{sec:regular_satisfying}

In this section we use \Cref{conj:upper_bound_DNF} to prove that regular enough DNFs are $(1/2,\eps)$-satisfying, where in light of \Cref{claim:satisfying_contains_disjoint} we care about $\eps=1/2$. To recall the definitions, a DNF $f$ is $(1/2,\eps)$-satisfying if for a uniformly chosen $x$,
$$
\Pr_x[f(x)=1] > 1-\eps.
$$
Define
$$
\gamma(w) = \sup \{\kappa: \exists \text{non-trivial } \kappa \text{-regular } w \text{-set system which is \underline{not} }(1/2,1/2) \text{-satisfying}\}.
$$
In other words, $\gamma(w)$ is the largest value, such that for any $\kappa > \gamma(w)$, if $\FF$ is a nontrivial $\kappa$-regular $w$-set system, then $\Pr_x[f_{\FF}(x)=1] > 1/2$.

We start by giving a lower bound on $\gamma(w)$, where the motivation is to help the reader gain intuition.

\begin{claim}
$\gamma(w) \ge \log w - O(1)$.
\end{claim}

\begin{proof}
We construct a non-trivial $\kappa$-regular $w$-set system which is not $(1/2,1/2)$-satisfying, for $\kappa=\log w - O(1)$.
Let $X_1,\ldots,X_w$ be disjoint sets, each of size $\kappa=\log w-c$ for a constant $c>0$ to be determined. Let $X=X_1 \cup \ldots \cup X_w$.
Let $\FF \subseteq \PP(X)$ be the $w$-set system of all sets $S$ that contain exactly one element from each set $X_i$. It is simple to verify that the uniform distribution over $\FF$ is $\kappa$-regular, and hence $\FF$ is $\kappa$-regular. Let $W \sim X_{1/2}$. Then
$$
\Pr[\exists S \in \FF, \; S \subseteq W] = \Pr[\forall i \in [w], |X_i \cap W| \ge 1] = (1-2^{-\kappa})^w = (1-c/w)^w \le \exp(-c).
$$
In particular, for $c \ge 1$ we get that $\FF$ is not $(1/2,1/2)$-satisfying.
\end{proof}

As we shall soon see, \Cref{conj:upper_bound_DNF} implies that the lower bound is not far from tight:
$$
\gamma(w) \le (\log w)^{O(1)}.
$$
It will be sufficient to assume
a slightly weaker version of \Cref{conj:upper_bound_DNF}, where we allow the size of $f_{upper}$ to be somewhat bigger.

\begin{conjecture}[Weaker version of \Cref{conj:upper_bound_DNF}]
\label{conj:upper_bound_weak}
Let $w \ge 2, \eps>0$. For any monotone width-$w$ DNF $f$ there exists a monotone width-$w$ DNF $f_{upper}$ such that
\begin{enumerate}[(i)]
\item $f_{upper}$ is a proper upper bound DNF for $f$.
\item $f_{upper}$ and $f$ are $\eps$-close.
\item $f_{upper}$ has size at most $((\log w)/\eps)^{cw}$ for some absolute constant $c>1$.
\end{enumerate}
\end{conjecture}

\begin{lemma}
\label{lemma:satisfy}
Assume \Cref{conj:upper_bound_weak} holds. Then there exists a constant $c_0>1$ such that the following holds.
For $w \ge 2, \eps>0$ let $\kappa_0(w,\eps)=((\log w)/\eps)^{c_0}$. Let $\FF$ be a non-trivial $\kappa$-regular $w$-set system for $\kappa=\kappa_0(w,\eps)$. Then $\FF$ is $(1/2, \eps)$-satisfying.
\end{lemma}

\begin{corollary}
\label{cor:gamma_upper}
$\gamma(w) \le \kappa_0(w,1/2) = (\log w)^{O(1)}$.
\end{corollary}

We prove \Cref{lemma:satisfy} in the remainder of this section.
We start with some simple claims that would serve as a base case for \Cref{lemma:satisfy} for $w=O(1)$.

\begin{claim}
\label{claim:regular_disjoint_base}
Let $r \ge 2$. Let $\FF$ be a non-trivial $\kappa$-regular $w$-set system, where $\kappa > w {r \choose 2}$. Then $\FF$ contains $r$ pairwise disjoint sets.
\end{claim}

\begin{proof}
Let $\D$ be a $\kappa$-regular distribution over $\FF$. Sample independently $S,S' \sim \D$.
The probability that $S,S'$ intersect is at most
$$
\Pr[|S \cap S'| \ge 1] \le \sum_{i \in S} \Pr[i \in S'] \le w / \kappa.
$$
Let $S_1,\ldots,S_r \sim \D$ be chosen independently. Then by the union bound, the probability
that two of them intersect is at most ${r \choose 2} w / \kappa < 1$. In particular, there exist $r$ pairwise disjoint sets in $\FF$.
As $\FF$ is non-trivial, the sets are non-empty, and hence distinct.
\end{proof}

\begin{claim}
\label{claim:regular_satisfying_base}
Let $\eps >0$. Let $\FF$ be a non-trivial $\kappa$-regular $w$-set system, where $\kappa = w (2^{w} \log(1/\eps))^2$. Then $\FF$ is $(1/2,\eps)$-satisfying.
\end{claim}

\begin{proof}
Assume $\FF \subseteq \PP(X)$. \Cref{claim:regular_disjoint_base} implies that $\FF$ contains $r=2^w \log(1/\eps)$ pairwise disjoint sets $S_1,\ldots,S_r$ of size at most $w$. Let $W \sim X_{1/2}$. Then
$$
\Pr[\exists S \in \FF, S \subseteq W] \ge
 \Pr[\exists i \in [r], S_i \subseteq W] \ge 1-(1-2^{-w})^r > 1-\eps.
$$
\end{proof}

\begin{proof}[Proof of \Cref{lemma:satisfy}]
We will need several properties from $\kappa_0$ in the proof.
To simplify notations, we shorthand $\kappa_0(w/2,\eps)$ for $\kappa_0(\lfloor w/2 \rfloor,\eps)$ throughout.
The constant $c > 1$ below is the absolute constant from \Cref{conj:upper_bound_DNF}. We need a constant $c'>1$ so that the following conditions are satisfied:
\begin{enumerate}[(i)]
\item $\kappa_0(w,\eps) \ge w (2^w \log(1/\eps))^2$ for $w=1,2$ and $\eps>0$.
\item $\kappa_0(w,\eps) \ge ((\log w)/\eps)^{12c}$ for $w \ge 3,\eps>0$.
\item $\kappa_0(w,\eps) \ge \kappa_0(w/2, \eps(1-1/\log w))+1$ for $w \ge 3,\eps>0$.
\end{enumerate}
One can check that the function $\tau(w,\eps) = (\log w) / \eps$ satisfies $\tau(w/2, \eps(1-1/\log w)) \le \tau(w,\eps)$, with equality
when $w$ is even. Thus taking $\kappa_0(w,\eps) = ((\log w)^2 / \eps)^{c'}$ satisfies the conditions for a large enough $c' \ge 12 c$. We then take $c_0 = 2c'$.

The proof of lemma \Cref{lemma:satisfy} is by induction on $w$.
The base cases are $w=1$ and $w=2$ which follow from \Cref{claim:regular_satisfying_base} and condition (i) on $\kappa_0$. Thus, we assume from now that $w \ge 3$. We need to prove that for $f=f_{\FF}$ we have
$$
\Pr[f(x)=0] < \eps.
$$

Let $\gamma = \eps/\log w$ and assume that $\FF$ is $\kappa$-regular for
$\kappa=\kappa_{0}(w,\eps)$. Let $\FF_1=\{S \in \FF: |S| \ge w/2\}$ and let $f_1=f_{\FF_1}$ be the corresponding DNF.
Applying \Cref{conj:upper_bound_weak} to $f_1$ with error parameter $\gamma$,
we obtain that there exists a $\gamma$-approximate proper upper bound DNF $f_2$ for $f_1$ of size $s=((\log w)/\gamma)^{cw} \le
((\log w)/\eps)^{2cw}$. Let $\FF_2$ be the corresponding set system to $f_2$,
and observe that $\FF_2$ is a proper upper bound set system for $\FF_1$.
Let $\FF_3 = (\FF \setminus \FF_1) \cup \FF_2$ and let $f_3=f_{\FF_3}$ be the corresponding DNF. Then
$$
\Pr[f(x)=0] \le \Pr[f_3(x)=0] + \left(\Pr[f_2(x)=0] - \Pr[f_1(x)=0]\right) \le \Pr[f_3(x)=0] + \gamma.
$$
We may assume without loss of generality that $\FF_3$ is non-trivial, as otherwise $f_3 \equiv 1$, hence $\Pr[f(x)=0] \le \gamma$ and we are done.

Next, observe that $\FF_3$ is a proper upper bound set system for $\FF$.
As we assume that $\FF$ is $\kappa$-regular, then by \Cref{claim:regular_upper} we obtain that $\FF_3$ is also
$\kappa$-regular. Let $\D$ be a $\kappa$-regular distribution supported on $\FF_3$. Let $\FF_4 = \{S \in \FF_3: |S| \ge w/2\}$,
where $\FF_4 \subseteq \FF_2$. As each set $S \in \FF_4$ has size $|S| \ge w/2$ then, since $D$ is $\kappa$-regular, we have
$$
\D(S) \le \kappa^{-w/2}.
$$
Summing over all $S \in \FF_4$ we obtain that
$$
\D(\FF_4) \le |\FF_4| \cdot \kappa^{-w/2}
\le |\FF_2| \cdot \kappa^{-w/2}
\le \left( \left(\frac{\log w}{\eps}\right)^{2c} \kappa^{-1/2} \right)^{w}.
$$
We would need that $\D(\FF_4) \le 1/\kappa$. As $w \ge 3$, this follows from condition (ii) on $\kappa_0$.
Let $\FF_5=\FF_3 \setminus \FF_4$. Then $\FF_5$ is a $(w/2)$-set system.
Note that $\FF_5$ is non-trivial since $\FF_4$ is a strict subset of $\FF_3$.
By \Cref{claim:regular_lower} $\FF_5$ is $\kappa'$-regular
for
$$
\kappa' = \kappa \cdot \D(\FF_5) = \kappa (1 - \D(\FF_4)) \ge \kappa-1.
$$
Let $\eps' = \eps(1-1/\log w)$. Assumption (iii) on $\kappa_0$ gives that $\kappa_0(w/2, \eps') \le \kappa_0(w,\eps)-1$. Thus, $\FF_5$
is $\kappa_0(w/2,\eps')$-regular.
Applying the induction hypothesis, if we denote by $f_5$ the corresponding DNF for $\FF_5$, then
$$
\Pr[f_5=0] < \eps'.
$$
Finally, as $\FF_5 \subseteq \FF_3$ we have $\Pr[f_3(x)=0] \le \Pr[f_5(x)=0]$. Putting these together we obtain that
$$
\Pr[f(x)=0] \le \Pr[f_3(x)=0] + \gamma \le \Pr[f_5(x)=0]+\gamma < \eps' + \gamma = \eps (1 - 1/\log w) + \eps / \log w= \eps.
$$
\end{proof}

\section{Intersecting regular set systems}
\label{sec:intersecting}
As we showed in \Cref{claim:satisfying_contains_disjoint}, if $\FF$ is a $(1/r,1/r)$-satisfying set system, then it contains $r$ pairwise disjoint sets. However we only proved that a regular enough set system is $(1/2,1/2)$-satisfying so far. In this section, we prove that this is enough to show the existence of an $r$-sunflower for any constant $r$, and with a comparable condition of regularity. Our proof is based on a the study of regular intersecting set systems.

\begin{definition}[Intersecting set system]
A set system is intersecting if any two sets in it intersect. In other words, it does not contain two disjoint sets.
\end{definition}

\begin{definition}
For $w \ge 1,r \ge 2$ define
$$
\alpha(w, r) = \sup \{\kappa: \exists \text{non-trivial }\kappa \text{-regular } w \text{-set system without } r \text{ pairwise disjoint sets}\}.
$$
It will be convenient to shorthand $\beta(w) = \alpha(w,2)$, which can equivalently be defined as
$$
\beta(w) = \sup \{\kappa: \exists \text{non-trivial } \kappa \text{-regular intersecting } w \text{-set system}\}.
$$
\end{definition}

\begin{claim}
$\alpha(w+1,r) \ge \alpha(w,r)$ and $\alpha(w,r+1) \ge \alpha(w,r)$ for all $w \ge 1, r \ge 2$.
\end{claim}

\begin{proof}
The first claim follows by our definition that a $w$-set system is a set system where all sets have size at most $w$. In particular, any $w$-set system is also a $(w+1)$-set system and hence $\alpha(w+1,r) \ge \alpha(w,r)$. The second claim holds since a set system that does not contain $r$ pairwise disjoint sets, also does not contain $r+1$ pairwise disjoint sets.
\end{proof}

We start by showing that upper bounds on $\alpha(w,r)$ directly translate to upper bounds on sunflowers. This is reminiscent to the original proof of Erd\H{o}s and Rado \cite{ErdosR1960}.

\begin{claim}
\label{claim:alpha_w_r_sunflower}
Let $\FF$ be a non-redundant $w$-set system of size $|\FF| > \alpha(w, r)^w$. Then $\FF$ contains an $r$-sunflower.
\end{claim}

\begin{proof}
Note that $\FF$ is also non-trivial by our assumptions.
The proof is by induction on $w$. If $\FF$ contains $r$ pairwise disjoint sets then we are done.
Otherwise, $\FF$ is not $\kappa$-regular for any $\kappa>\alpha(w,r)$.
In particular, the uniform distribution over $\FF$ is not $\kappa$-regular. This implies that there exists a nonempty set
$T$ of size $|T|=t \ge 1$ such that
$$
\FF' = \{S \setminus T: S \in \FF, T \subseteq S\}
$$
has size $|\FF'| \ge |\FF| \kappa^{-t} > \alpha(w,r)^{w-t} \ge \alpha(w-t,r)^{w-t}$. Observe that $\FF'$ is also non-redundant.
By induction, $\FF'$ contains an $r$-sunflower $S_1 \setminus T, \ldots, S_r \setminus T$. Hence $S_1,\ldots,S_r$ is a sunflower in $\FF$.
\end{proof}

\begin{corollary}
\label{cor:alpha_w_r_sunflower_general}
Let $\FF$ be a $w$-set system of size $|\FF| > (2\alpha(w, r))^w$. Then $\FF$ contains an $r$-sunflower.
\end{corollary}

\begin{proof}
Let $\FF' \subset \FF$ be the sub set system of all the \emph{maximal} sets in $\FF$,
$$
\FF' = \{S \in \FF: \text{there is no } S' \in \FF \text{ with } S \subsetneq S'\}.
$$
Note that $\FF'$ is a non-redundant $w$-set system of size $|\FF'| \ge |\FF|/2^w$, and hence by \Cref{claim:alpha_w_r_sunflower} it contains an $r$-sunflower.
\end{proof}

The main lemma we prove in this section is that upper bounds on $\beta$ imply upper bounds on $\alpha$.

\begin{lemma}
\label{lemma:alpha_beta}
For all $w \ge 1, r \ge 3$ it holds that
$\alpha(w, r ) \leq r 2^{r+1} \beta(wr)^r$.
\end{lemma}

Before proving \Cref{lemma:alpha_beta}, we first prove some upper and lower bounds on $\beta(w)$. Although these are not needed in the proof of \Cref{lemma:alpha_beta}, we feel that they help gain intuition on $\beta(w)$.

\begin{claim}
\label{claim:beta_upper}
$\beta(w) \leq w$.
\end{claim}

\begin{proof}
Apply \Cref{claim:regular_disjoint_base} for $r=2$.
\end{proof}

It is easy to construct examples that show that $\beta(w)>1$; for example, the family of all sets of size $w$ in a universe of size $2w-1$ is intersecting and $((2w-1)/w)$-regular. The following example shows that $\beta(w)$ is super-constant.

\begin{claim}
\label{claim:beta_lower}
$\beta(w) \geq \Omega\left(\frac{ \log w}{\log\log w}\right)$.
\end{claim}

\begin{proof}
We construct an example of a $\kappa$-regular intersecting $w$-set system for $\kappa=\Omega(\log w / \log \log w)$.
Let $t \le w/2$ to be optimized later and set $m=w-t+1$. Let $X_1,\ldots,X_m$ be disjoint sets of size $t$ each, and let $X=X_1 \cup \ldots \cup X_m$. Consider the set system $\FF$ of all sets $S \subseteq X$ of the following form:
$$
\FF = \{S \subseteq X: \exists i \in [m], \; X_i \subseteq S,\;\forall j \ne i,\; |X_j \cap S|=1\}.
$$
Observe that $\FF$ is an intersecting $w$-set system.

Let $\D$ be the uniform distribution over $\FF$. We show that $\D$ is $\kappa$-regular, and hence $\FF$ is $\kappa$-regular.
There are two extreme cases: for sets $T$ of size $|T|=1$ we have
$$
\Pr_{S \sim \D} \left[T \subseteq S\right] = \frac{1}{m} + \left(1 - \frac{1}{m} \right) \frac{1}{t} \le \frac{2}{t}.
$$
For sets $T=X_i$ we have
$$
\Pr_{S \sim \D} \left[X_i \subseteq S\right] = \frac{1}{m}.
$$
One can verify that these are the two extreme cases which control the regularity, and hence $\FF$ is $\kappa$-regular for
$$
\kappa = \min(t/2, m^{1/t}).
$$
Setting $t=\Theta(\log w / \log \log w)$ gives $\kappa=\Theta(\log w / \log \log w)$.
\end{proof}

We conjecture that this is essentially tight. In fact, by \Cref{claim:satisfying_contains_disjoint} we have that
\[
\beta(w) \leq \gamma(w)
\]
As we proved, \Cref{conj:upper_bound_weak} implies $\gamma(w) = (\log w)^{O(1)}$, thus it also implies $\beta(w) = (\log w)^{O(1)}$.

\begin{proof}[Proof of \Cref{lemma:alpha_beta}]
For $w,r \ge 1$ define
$$
\eta(w,r) = r 2^{r+1} \beta(wr)^r.
$$
We will first prove that
$$
\alpha(w,2r) \le \max\left(\eta(w,r), 2 \alpha(w,r)\right)
$$
and then that this implies the bound
$$
\alpha(w,r) \le \eta(w,r).
$$

Let $\FF$ be a non-trivial $\kappa$-regular $w$-set system which does not contain $2r$ pairwise disjoint sets, where $\kappa > \max\left(\eta(w,r), 2 \alpha(w,r)\right)$. We will show that this leads to a contradiction.

Let $\D$ be the corresponding $\kappa$-regular distribution on $\FF$. Let $\FF' \subseteq \FF$ be any sub set-system with $\D(\FF') \ge 1/2$.
\Cref{claim:regular_lower} then implies that $\FF'$ is $(\kappa/2)$-regular. By our choice of $\kappa$, $\kappa/2 > \alpha(w,r)$, and hence $\FF'$ contains $r$ pairwise disjoint sets.

More generally, consider the following setup. Let $\D':\FF \to \mathbb{R}_{\ge 0}$
with $\D'(S) \le \D(S)$ for all $S \in \FF$. Define $\FF'=\{S: \D'(S)>0\}$
and $\D'(\FF) = \sum \D'(S)$. As long as $\D'(\FF) \ge 1/2$ we are guaranteed that $\FF'$ is $(\kappa/2)$-regular,
and hence contains $r$ pairwise disjoint sets. Consider the following process:
\begin{enumerate}
    \item Initialize $D_0(S)=\D(S)$ for all $S \in \FF$ and $i=0$.
    \item As long as $D_i(\FF) \ge 1/2$ do:
\begin{enumerate}
    \item Let $\FF_i = \{S: D_i(S)>0\}$.
    \item Find $r$ pairwise disjoint sets $S_{i,1},\dots,S_{i,r} \in \FF_i$.
    \item Let $\delta_i = \min(D_i(S_{i,1}),\ldots,D_i(S_{i,r}))$.
    \item Set $D_{i+1}(S)=D_i(S) - \delta_i$ if $S \in \{S_{i,1},\ldots,S_{i,r}\}$, and $D_{i+1}(S)=D_i(S)$ otherwise.
    \item Set $i\leftarrow i + 1$
\end{enumerate}
\end{enumerate}

Assume that the process terminates after $m$ steps. Let $W_i = S_{i,1} \cup \ldots \cup S_{i,r}$, which by construction is a non-empty set
of size at most $wr$. Note that as we assume that $\FF$ does not contain $2r$ pairwise disjoint sets, we obtain that $W_1,\ldots,W_m$
must be an intersecting set system (possibly with some repeated sets). Let $\delta = \sum \delta_i$. As $D_{i+1}(\FF) = D_i(\FF) - \delta_i r$, and as we terminate when $D_m(\FF)<1/2$,
we have
$$
\delta \ge 1/2r.
$$
Let $\FF^*=\{W_1,\ldots,W_m\}$, namely taking each set exactly once.
As it may be the case that $W_1,\ldots,W_m$ are not all distinct, we only know that $|\FF^*| \le m$. Consider the distribution $D^*$ on $\FF^*$ given by $D^*(W) = \frac{1}{w}\sum_{i: W_i=W} w_i$. Then as $\FF^*$ is a non-trivial intersecting set system,
we obtain that $D^*$ cannot be $\beta$-regular for $\beta=\beta(wr)$.

Thus, there exists a nonempty set $T$ of size $|T|=t \ge 1$ such that
$$
\sum_{W \in \FF^*: \; T \subseteq W} D^*(W) \ge \beta^{-t}.
$$
This implies that if we denote $I = \{i \in [m]:\; T \subseteq W_i\}$ then
$$
\sum_{i \in I} w_i \ge w \beta^{-t} \ge \frac{1}{2r \beta^t}.
$$
Next, consider some $i \in I$. Recall that $W_i$ is the union of pairwise disjoint sets $S_{i,1},\ldots,S_{i,r} \in \FF$.
In particular, there must exist $j_i \in [r]$ such that $|T \cap S_{i,j_i}| \ge |T|/r$. We denote $T_i = T \cap S_{i,j_i}$.
As the number of possibles subsets of $T$ is $2^{|T|}$, there must exist $T^* \subseteq T$ such that
$$
\sum_{i \in I:\; T_i=T^*} w_i \ge
2^{-t}
\sum_{i \in I} w_i \ge
\frac{1}{2r (2\beta)^t}.
$$
In particular, $|T^*| \ge |T|/r$ and
$$
\sum_{i \in I:\; T^* \subseteq S_{i,j_i}} w_i \ge \frac{1}{2r (2\beta)^t}.
$$

It may be that the list of $S_{i,j_i}$ contains repeated sets (namely, that $S_{i,j_i}=S_{i',j_{i'}}$ for some $i \ne i'$).
For each $S \in \FF$ let $I(S)=\{i \in I: S_{i,j_i}=S\}$. In particular, $I(S)$ is not empty only for sets $S$ with $T^* \subseteq S$. We can rewrite the sum as
$$
\sum_{i \in I:\; T^* \subseteq S_{i,j_i}} w_i = \sum_{S \in \FF:\; T^* \subseteq S} \sum_{i \in I(S)} w_i.
$$

Next, fix some $S \in \FF$ with $T^* \subseteq S$ and consider the internal sum. Recall that $w_i = D_{i}(S)-D_{i+1}(S)$, and hence the sum is a telescopic sum and can be bounded by
$$
\sum_{i \in I(S)} w_i \le D_{0}(S) - D_m(S) \le D(S).
$$
We thus obtain that
$$
\sum_{S \in \FF: T^* \subseteq S}  \D(S) \ge \sum_{i \in I:\; T^* \subseteq S_{i,j_i}} w_i \ge \frac{1}{2r (2 \beta)^t}.
$$

Recall that $\D$ is $\kappa$-regular. We can upper bound $\kappa$ by
$$
\kappa \le \left( 2r (2 \beta)^t \right)^{1/|T^*|} \le \left( 2r (2 \beta)^t \right)^{r/t} \le 2r (2 \beta)^r = \eta(w,r).
$$
Putting everything together, we get
$$
\alpha(w,2r) \le \max\left(\eta(w,r), 2 \alpha(w,r) \right).
$$
To conclude the proof, note that if $r$ is a power of two then by induction and our choice of $\eta$ we have
$$
\alpha(w,2r) \le \max\left(\eta(w,r), 2 \eta(w,r/2), 4 \eta(w,r/4), \ldots \right) = \eta(w,r).
$$
Thus for a general $r$, if $r \le s \le 2r$ is the smallest power of two that upper bounds $r$ then
$$
\alpha(w,r) \le \alpha(w,s) \le \eta(w,s/2) \le \eta(w,r).
$$

\end{proof}

\begin{proof}[Proof of \Cref{thm:main}]
We put all the pieces together. Let $w \ge 2$. Assume \Cref{conj:upper_bound_weak} holds. \Cref{lemma:satisfy} gives that
$$
\gamma(w) \le (\log w)^{c}
$$
for some constant $c \ge 1$. \Cref{claim:regular_disjoint_base} then gives that
$$
\beta(w) \le \gamma(w)
$$
and \Cref{lemma:alpha_beta} gives that
$$
\alpha(w,r) \le r 2^{r+1} \beta(wr)^r \le r 2^{r+1} (\log (wr))^{c r} \le (\log w)^{c_r}
$$
for some constant $c_r \ge 1$.
Finally, \Cref{claim:alpha_w_r_sunflower} shows that if $\FF$ is a $w$-set system of size $|\FF| \ge (\log w)^{c_r w}$ then $\FF$ contains an $r$-sunflower.
\end{proof}

\section{Further discussions}
Recall that $\beta(w)$ is the maximal $\kappa$ such that there exists an intersecting $\kappa$-regular $w$-set system.

\begin{conjecture}
\label{conj:beta_log}
$\beta(w) \le (\log w)^{O(1)}$.
\end{conjecture}

We would like to point out that in \Cref{conj:beta_log}, the assumption that the set system is intersecting cannot be replaced by a weaker assumption that it is almost intersecting, namely that most pairs of sets intersect. To see that, consider the following example.

\begin{example}
Let $\FF$ be the family of all sets of size $w$ in a universe of size $n=c w^2$. By choosing an appropriate constant $c>0$, we get that $99 \%$ of the sets $S,S' \in \FF$ intersect. However, $\FF$ is $(w/c)$-regular.
\end{example}

The following is an interesting family of examples, that might help shed light on \Cref{conj:beta_log}.

\begin{example}
\label{ex:subspace}
Let $\F_p$ be a finite field and $n \ge 1$. Let $V \subset \F_p^n$ be a linear subspace of dimension $k$.
Given a set of coordinates $I \subseteq [n]$, define $V_I = \{ (v_i)_{i \in I}: v \in V\}$ to be the subspace obtained by
restricting vectors $v \in V$ to coordinates $I$. We say that $V$ is \emph{$\alpha$-large} if
$$
\dim(V_I) \ge \alpha |I| \qquad \forall I \subseteq [n].
$$
In particular, this implies that $k \ge \alpha n$.

Next, we define a set system corresponding to a subspace.
Let $X=\{(i,a): i \in [n], a \in \F_p\}$. For any vector $v \in \F_p^n$ define its corresponding set
$$
S(v)=\{(i,v_i): i \in [n]\} \subset X.
$$
For a subspace $V \subset \F_p^n$ define the set system
$$
\FF(V) = \{S(v): v \in V\}.
$$
Observe that:
\begin{enumerate}[(i)]
\item $\FF(V)$ is an $n$-set system of size $p^k$.
\item For any $T \subseteq X$ it holds that
$|\{S \in \FF(V): T \subseteq S\}| \le p^{-\alpha|T|} |\FF|$.
Hence $\FF(V)$ is $\kappa$-regular for $\kappa=p^{\alpha}$.
\item $\FF(V)$ is intersecting iff any $v \in V$ contains at least one zero coordinate.
\end{enumerate}
If \Cref{conj:beta_log} holds and $p \ge (\log n)^c$ for some absolute constant $c>0$, then it must hold that $V$ contains a vector with no zero coordinates. This motivates the following problem.
\end{example}

\begin{problem}\label{prob:subspace}
Let $V \subset \F_p^n$ be a $\alpha$-large subspace. Prove that if $p \ge (\log n)^c$, for some $c=c(\alpha)$, then $V$ must contain
a vector with no zero coordinates.
\end{problem}

A previous version of this paper gave a more restricted version of \Cref{ex:subspace}, corresponding to the case when $V$ spans an MDS code. Namely, $\dim(V_I)=|I|$ for all $I \subseteq [n]$ with $|I| \le k$. Ryan Alweiss \cite{alweiss-personal} proved the analog of \Cref{prob:subspace}
for this case, in fact where $p \ge p_0(n/k)$.

\bibliographystyle{alpha}
\bibliography{reference}

\end{document}